\documentclass[12pt,titlepage]{amsart}
\usepackage{amsfonts, amssymb, amsmath,amsthm, mathdots, verbatim, pmat,pdflscape}
\newtheorem{theorem}{Theorem}[section]

\newtheorem{lemma}[theorem]{Lemma}
\newtheorem{corollary}[theorem]{Corollary}
\theoremstyle{remark}

\newtheorem{conjecture}[theorem]{Conjecture}

\theoremstyle{definition}

\DeclareMathOperator{\diag}{{\mathrm {diag}}} 	
\newcommand{\abs}[1]{|#1|}
\newcommand{\mb}[1]{\mathbf{#1}}

\newcommand{\Z}{{\mathbb Z}}

\newcommand{\F}{{\mathbb F}}
\newcommand\coolleftbrace[2]{%
#1\left\{\vphantom{\begin{matrix}%
#2 \end{matrix}}\right.}
\begin{document}
\title{The Smith group of the hypercube graph}
\author{David B. Chandler, Peter Sin$^*$ and Qing Xiang}
\address{6 Georgian Circle, Newark, DE 19711\\USA}
\address{Department of Mathematics\\University of Florida\\ P. O. Box 118105\\ Gainesville, FL 32611\\ USA}
\address{Department of Mathematical Sciences, University of Delaware, Newark, DE 19716\\USA}
\date{}
\thanks{$^*$This work was partially supported by a grant from the Simons Foundation (\#204181 to Peter Sin)}
\date{}
\begin{abstract} The $n$-cube graph is the graph on the vertex set
of $n$-tuples of $0$s and $1$s, with two vertices joined by
an edge if and only if the $n$-tuples differ in exactly one component.
We compute the Smith group of this graph, or, equivalently, the
elementary divisors of an adjacency matrix of the graph.
\end{abstract}
\maketitle
\section {Introduction}
Let $Q_n$ be the $n$-cube graph, with vertex set $\{0,1\}^n$ and
two vertices joined if they differ in one component. In the language of association
schemes, $Q_n$ is the distance 1 graph of the binary Hamming scheme.

It is of interest to compute linear algebraic invariants of a graph, such as 
its eigenvalues and the invariant factors of an adjacency matrix  or
Laplacian matrix. In the case of $Q_n$, previous work includes
\cite{Bai} and \cite{DJ}, where many of these invariants have been computed 
and some conjectures made about others. Here we shall consider the
Smith group. If $X$ is an $m\times n$ integral matrix, then the {\it Smith
group} of $X$ is the abelian group defined as the quotient of $\Z^m$ by the subgroup 
spanned by the columns $X$; that is, the abelian group whose invariant factor decomposition
is given by the Smith normal form of $X$. 
If  $A$ is the adjacency matrix (with respect to any ordering
of the vertices) of a graph, then the Smith group of the graph
is defined as the Smith group of $A$, and does not depend on the ordering
on the vertices.
 
We recall that two integral matrices $X$ and $Y$ are  \emph{integrally equivalent}
if there exist unimodular integral matrices $U$ and $V$ such that
\begin{equation}\label{UXV}
UXV=Y.
\end{equation}
As is well known, $X$ and $Y$ are integrally equivalent if
and only if  $Y$ can be obtained from  $X$ by a finite sequence of integral 
unimodular row and column operations.
A {\it diagonal form} for $X$
is a matrix integrally equivalent to $X$ that has nonzero entries only
on the leading diagonal. The Smith normal form of $X$ is one such diagonal form.
Another way to describe the Smith group is in terms of the  $p$-elementary divisors of $X$ with respect to primes $p$.
Any diagonal form for $X$  gives a cyclic decomposition of the Smith group,
so in a certain sense, the various  diagonal forms  carry  the same information
as the list of $p$-elementary divisors as $p$ varies over all primes.  

The notion of integral equivalence can be generalized to $\Z_{(p)}$-equivalence, where $\Z_{(p)}$
is the ring of $p$-local integers, by requiring that the matrices  $U$ and $V$
appearing in (\ref{UXV}) be invertible over $\Z_{(p)}$. We can also consider the
$p$-elementary divisors of any matrix $X$ with entries in $\Z_{(p)}$.
If $X$ happens to have integer entries then its $p$-elementary
divisors are the same whether it is considered as a matrix over $\Z$
or $\Z_{(p)}$

Let $A$ be an adjacency matrix for $Q_n$.
It was proved in \cite{DJ} that for every odd prime $p$, $A$ is $\Z_{(p)}$-equivalent to
the diagonal matrix of  the eigenvalues (all of which are integers).

When $n$ is odd, all the eigenvalues are odd integers, so $\det(A)$
is an odd integer. Thus, the eigenvalue  matrix is  a diagonal form when $n$ is odd.

 When $n$ is even,  
 there remains the problem of finding the $2$-elementary divisors 
 of $A$.  A conjecture for the multiplicity of each power $2^e$ as
a $2$-elementary divisor was stated in \cite{DJ}. (See Conjecture~\ref{conjDJ} below.)
The purpose of this paper is to give a proof of this conjecture. 
As a consequence of the conjecture, we obtain the following
description of the Smith group of $Q_n$ when $n$ is even.
 
\begin{theorem}\label{Qdiag} Suppose that $n=2m$ is even. Then the adjacency matrix
 $A$ of the $n$-cube $Q_n$ is  integrally equivalent to a diagonal matrix 
 with $\tbinom{n}{m}$ diagonal entries equal to zero and 
 whose nonzero diagonal entries are the integers $k=1$,$2$,\dots $m$, in which the multiplicity of $k$ is $2\tbinom{n}{m-k}$.
  
\end{theorem}

\section{Inclusion of subsets of a finite set}
Let $n$ be a positive integer and 
$X=\{1,2,\ldots, n\}$. For brevity, we shall use the term
\emph{$k$-subsets} for the subsets of $X$ of size $k$.
For $k\leq n$, let $M_k$ denote the free $\Z$-module on the set of $k$-subsets
and for $t$, $k\leq n$ let
$$
\eta_{t,k}: M_t\to M_k
$$
be the incidence map, induced by inclusion. Thus, if $t\leq k$ 
a $t$-subset is mapped to the sum of all $k$-subsets containing it,
while if $t\geq k$ the image of a $t$-subset is the sum of
all $k$-subsets which it contains.

For each $k\leq n$, if we fix ordering on the $k$-subsets, 
we can think of elements of  $M_k$ as row vectors.  Let $W_{t,k}$  denote the  $\tbinom{n}{t}\times\tbinom{n}{k}$  matrix of $\eta_{t,k}$ with respect to these ordered bases of $M_t$ and $M_k$.

\section{Canonical Bases for subset modules}\label{canon}
The notion of the \emph{rank} of a subset was introduced by Frankl \cite{Frankl}.
We shall only need the concept of a $t$-subset of rank $t$, for $t\leq \frac{n}{2}$.
let $T=\{i_1,i_2,\ldots, i_t\}\subseteq X$, with the elements in increasing order.
Then $T$ has rank $t$ if and only if $i_j\geq 2j$ for all $j=1$,\dots,$t$.
A $t$-subset is of rank $t$ if an only if it is
the set of entries in the second row of a standard 
Young tableau of shape $[n-t,t]$. This is one way to see that
the number of $t$-subsets of rank $t$ is $\tbinom{n}{t}-\tbinom{n}{t-1}$.

Assume $0\leq j\leq k\leq \frac{n}{2}$.
Let $E_{j,k}$ denote the $[\tbinom{n}{j}-\tbinom{n}{j-1}]\times \tbinom{n}{k}$
submatrix of $W_{j,k}$ formed from the 
the rows labeled by $j$-subsets of rank $j$, and let
$E_k$ be the $\tbinom{n}{k}\times \tbinom{n}{k}$ matrix formed by
stacking the $E_{j,k}$, $0\leq j\leq k$,  with $j$ increasing as we move down the
matrix $E_k$. 


Wilson \cite{Wilson} found a diagonal form for $W_{t,k}$. We shall
state his result for $t\leq k\leq n/2$.
There exist unimodular matrices $U_{t,k}$
and $V_{t,k}$ such that  
\begin{equation}\label{UV}
U_{t,k}W_{t,k}V_{t,k}= D_{t,k},
\end{equation}
where the diagonal form $D_{t,k}$ has diagonal entries
$\tbinom{k-j}{t-j}$,  with multiplicity $\tbinom{n}{j}-\tbinom{n}{j-1}$, for
$j=0$, \dots, $t$.  
Bier \cite{Bier} refined Wilson's results, showing that
we can take $U_{t,k}=E_t$ for all $k$ and $V_{t,k}={E_k}^{-1}$ for all $t$. 
The additional uniformity will be important for us.

\begin{theorem}\label{Bier}\cite{Bier}
Assume $k\leq \frac{n}{2}$. Then the matrix $E_k$ is unimodular.
Furthermore, for  $t\leq k$, we have
\begin{equation}
E_t W_{t,k} {E_k}^{-1} = D_{t,k},
\end{equation}
where $D_{t,k}$ is Wilson's diagonal form.
\end{theorem}

We shall refer to the basis of $M_k$ corresponding to
the rows of $E_k$ as the \emph{canonical basis} of $M_k$.
It consists of all vectors of the form $\eta_{j,k}(J)$, where
$0\leq j\leq k$ and $J$ is a $j$-subset of rank $j$.

\section{The $n$-cube}\label{ncube}
Let $Q_n$ denote the $n$-cube graph. The vertex set of $Q_n$ is $\{0,1\}^n$
and $(a_1,\ldots,a_n)$ is adjacent to $(b_1,\ldots,b_n)$ if and only if
there is exactly one index $j$ with $a_j\neq b_j$. There is clearly a bijection
of the vertex set with the set of subsets of $X=\{1,2,\dots,n\}$,
under which a vertex corresponds to the subset of indices where the vertex has entry $1$.
We use this bijection and our fixed ordering of $k$-subsets
for $k\leq n$ to order the vertices of $Q_n$, taking the subsets in order of increasing 
size. Let $A$ denote the adjacency matrix, with respect to this ordering. 

Next we review the results of \cite{DJ}.
By viewing the vertex set of $Q_n$ as $\F_2^n$, and transforming $A$ by the 
character table of the additive group $\F_2^n$ 
the eigenvalues of $Q_n$ are found to be $n-2\ell$, 
with multiplicity $\tbinom{n}{\ell}$ for $0\leq\ell\leq n$.

When $n$ is odd, we see that $\det A$ is odd; hence all elementary divisors are odd.
Then since $\F_2^n$ is an abelian $2$-group, the same
discrete Fourier transform method yields the elementary divisors.
When $n$ is even, one still obtains the $p$-elementary divisors for all
odd $p$. The $2$-elementary divisors were not computed in \cite{DJ}, but
the following conjecture was stated: 

\begin{conjecture}\label{conjDJ}\cite[4.4.1]{DJ} Suppose $n$ is even. Then the multiplicity
of $2^i$ as a $2$-elementary divisor of $A$ is equal to the number of eigenvalues
of $A$ whose exact $2$-power divisor is $2^{i+1}$. 
\end{conjecture}

\section{Bases for the free module on $Q_n$ and matrix representations
of adjacency} 
Let $\Z^{Q_n}$ denote the free $\Z$-module on the set of vertices of $Q_n$. 
The matrix $A$  can be viewed as an endomorphism of $\Z^{Q_n}$,
sending a vertex to the sum of all adjacent vertices.
It is important for us to adopt a slightly different point of view. We can
think of $\Z^{Q_n}$ as the ring of $\Z$-valued functions on the set of vertices of $Q_n$. 
Then the matrix $A$ defines the map $\alpha$ such that for any function
$f\in\Z^{Q_n}$ we have
\begin{equation}
\alpha(f)(a_1,\ldots,a_n)=\sum_{i=1}^n f(a_1,\ldots a_{i-1},1-a_i,a_{i+1},\ldots,a_n),
\end{equation}
for $(a_1,\ldots,a_n)\in Q_n$.
If we further regard the set $\{0,1\}^n$ of vertices of $Q_n$ as a subset of
$\Z^n$,  then functions are restrictions of polynomials and
we have a ring isomorphism of $\Z^{Q_n}$ with
$$
\Z[X_1,\ldots,X_n]/({X_i}^2-X_i,\, 1\leq i\leq n).
$$
Now a different
natural basis becomes evident, namely the set of
monomials $X_I=\prod_{i\in I}X_i$, for $I\subseteq X=\{1,\ldots,n\}$.
With respect to the monomial basis we have
\begin{equation}
\alpha(X_I)= \sum_{i\in I} (X_{I\setminus\{i\}}-X_I) + \sum_{i\notin I} X_I
=(n-2\abs{I})X_I + \sum_{\begin{smallmatrix}J\subset I\\ \abs{J}=\abs{I}-1\end{smallmatrix}
}X_J. 
\end{equation}
Therefore, if we order monomials in the same way as we ordered
subsets, the matrix of $\alpha$ with respect to this basis has the form

\begin{equation*}
\tilde A=
\begin{pmat}[{|||||||}]
 nI&W_{0,1}&0&0&0&\hdots&0&0\cr\-
 0&(n-2)I&W_{1,2}&0&0&\hdots&0&0\cr\-
 0&0&(n-4)I&W_{2,3}&0&\hdots&0&0\cr\-
 0&0&0&(n-6)I&W_{3,4}&\hdots&0&0\cr\-
 \vdots&\vdots&\ddots&\ddots&\ddots&\ddots&\vdots&\vdots\cr\-
 0&0&\hdots&0&0&-(n-4)I&W_{n-2,n-1}&0\cr\-
 0&0&\hdots&0&0&0&-(n-2)I&W_{n-1,n}\cr\-
 0&0&\hdots&0&0&0&0&-nI\cr
 \end{pmat}.
\end{equation*}

Assume from now on that $n=2m$ is even.

\begin{equation}\label{MNmatrix}
\tilde A=\begin{bmatrix}M&0\\0&N
\end{bmatrix}
\end{equation}

\begin{equation*}
M=
\begin{pmat}[{|||||||}]
 nI&W_{0,1}&0&0&0&\hdots&0&0\cr\-
 0&(n-2)I&W_{1,2}&0&0&\hdots&0&0\cr\-
 0&0&(n-4)I&W_{2,3}&0&\hdots&0&0\cr\-
 0&0&0&(n-6)I&W_{3,4}&\hdots&0&0\cr\-
 \vdots&\vdots&\vdots&\ddots&\ddots&\ddots&\vdots&\vdots\cr\-
 0&\hdots&\hdots&\hdots&\ddots&4I&W_{m-2,m-1}&0\cr\-
 0&\hdots&\hdots&\hdots&\hdots&0&2I&W_{m-1,m}\cr
 \end{pmat}.
\end{equation*}

\begin{equation*}
N=\begin{pmat}[{||||||}]
W_{m,m+1}&0&\hdots&0&0&0&0\cr\-
-2I&W_{m+1,m+2}&\hdots&0&0&0&0\cr\-
0&-4I&\ddots&\vdots&\vdots&\vdots&\vdots\cr\-
\vdots&\ddots&\ddots&W_{n-4,n-3}&0&0&0\cr\-
\vdots&\vdots&\ddots&-(n-6)I&W_{n-3,n-2}&0&0\cr\-
\vdots&\vdots&\hdots&0&(-(n-4)I&W_{n-2,n-1}&0\cr\-
\vdots&\vdots&\hdots&0&0&-(n-2)I&W_{n-1,n}\cr\-
0&0&\hdots&0&0&0&-nI\cr
 \end{pmat}.
\
\end{equation*}

Due to the block form (\ref{MNmatrix}) of $\tilde A$, the multiplicity
of a prime power as an elementary divisor of  $\tilde A$
is the sum of its multiplicites in $M$ and $N$.

Up to now the choice of orderings on the $j$-susbsets, $0\leq j\leq n$
used in the definition of the inclusion matrices $W_{t,k}$ 
has been an arbitrary (but fixed) one. Any choice would result in matrices of the form
$M$ and $N$ as above, but the rows and columns of the submatrices
$W_{t,k}$ would be permuted. 

Now we shall specify these orderings more carefully.
The matrix $M$ involves only the matrices $W_{t,k}$ with $0\leq t<k\leq m$,
while the matrix $N$ involves only the matrices $W_{t,k}$ with
$m\leq t<k\leq n$. We start from the arbitrary but
fixed ordering on the $j$-subsets with $0\leq j\leq m$ that led to
the matrix $M$. Then for $0\leq j<m$ we choose the ordering of $(n-j)$-subsets
to be the order induced by the complementation map. In this way
we have specified an ordering on the $j$-subsets, for all $j$. 
Finally we wish to consider a \emph{second} ordering on $m$-subsets,
namely, the ordering defined from the given ordering by complementation.
We use the first ordering on $m$-sets to define  the submatrix $W_{m-1,m}$ of $M$
and the second ordering to define the submatrix $W_{m,m+1}$ of $N$.
From the block form (\ref{MNmatrix}), we see that the matrix $\tilde A$
thus constructed differs from the one in which the same ordering on
$m$-sets is used for both $W_{m-1,m}$ and $W_{m,m+1}$ 
only by a permutation of the rows in the first row-block of $N$, so
the two matrices would be integrally equivalent.

The reason we have been careful to choose the ordering as above is that
we now have, for $0\leq t<k\leq m$, 
\begin{equation}\label{transpose}
W_{n-k,n-t}=W^t_{k,t}.
\end{equation}

If we reverse the order of the block-rows and block-columns of $N$
and then take the transpose, we obtain

\begin{equation*}
\begin{aligned}
N'&=
\begin{pmat}[{|||||||}]
 -nI&W^t_{n-1,n}&0&0&0&\hdots&0&0\cr\-
 0&-(n-2)I&W^t_{n-2,n-1}&0&0&\hdots&0&0\cr\-
 0&0&-(n-4)I&W^t_{n-3,n-2}&0&\hdots&0&0\cr\-
 0&0&0&-(n-6)I&W^t_{n-4,n-3}&\hdots&0&0\cr\-
 \vdots&\vdots&\vdots&\ddots&\ddots&\ddots&\vdots&\vdots\cr\-
 0&\hdots&\hdots&\hdots&\ddots&-4I&W^t_{m+1,m+2}&0\cr\-
 0&\hdots&\hdots&\hdots&\hdots&0&-2I&W^t_{m,m+1}\cr
 \end{pmat}\\
&=
\begin{pmat}[{|||||||}]
 -nI&W_{0,1}&0&0&0&\hdots&0&0\cr\-
 0&-(n-2)I&W_{1,2}&0&0&\hdots&0&0\cr\-
 0&0&-(n-4)I&W_{2,3}&0&\hdots&0&0\cr\-
 0&0&0&-(n-6)I&W_{3,4}&\hdots&0&0\cr\-
 \vdots&\vdots&\vdots&\ddots&\ddots&\ddots&\vdots&\vdots\cr\-
 0&\hdots&\hdots&\hdots&\ddots&-4I&{W_{m-2,m-1}}&0\cr\-
 0&\hdots&\hdots&\hdots&\hdots&0&-2I&W_{m-1,m}\cr
 \end{pmat}.
\end{aligned}
\end{equation*}

Thus, $N'$ differs from $M$ only by the sign of the diagonal
entries. In fact, to see that $N'$ is integrally
equivalent to $M$, we perform the following simple sequence of unimodular operations.
First mutliply the first block-column by -1, then multiply the second block-row
by -1, then the third block-column, etc.,  until we reach the bottom-right
of the matrix, at which point $N'$ has been converted to $M$.

We have established the following reduction.

\begin{lemma}\label{half} Let $M$ and $N$ be the matrices in (\ref{MNmatrix}).
Then $M$ and $N^t$ are integrally equivalent.
In particular the multiplicity of an elementary divisor of $\tilde A$ 
(and hence of $A$) is twice its multplicity as an elementary divisor of $M$. 
\end{lemma}

From now, we focus on the matrix $M$.
Let $E_j$ for $0\leq j\leq m$ be defined as  in \S~\ref{canon}
and for $0\leq k\leq m$, set
\begin{equation*}
E(k)=\diag({E_0,E_1,\ldots,E_k}).
\end{equation*}

Then from Theorem~\ref{Bier} we immediately obtain:

\begin{equation}\label{Jmatrix}
E(m-1)\cdot M\cdot E(m)^{-1}=
\begin{pmat}[{|||||||}]
 nI&D_{0,1}&0&0&0&\hdots&0&0\cr\-
 0&(n-2)I&D_{1,2}&0&0&\hdots&0&0\cr\-
 0&0&(n-4)I&D_{2,3}&0&\hdots&0&0\cr\-
 0&0&0&(n-6)I&D_{3,4}&\hdots&0&0\cr\-
 \vdots&\vdots&\vdots&\ddots&\ddots&\ddots&\vdots&\vdots\cr\-
 0&\hdots&\hdots&\hdots&\ddots&4I&D_{m-2,m-1}&0\cr\-
 0&\hdots&\hdots&\hdots&\hdots&0&2I&D_{m-1,m}\cr
 \end{pmat}.
\end{equation}

For example when $n=4$, we have

\begin{equation*}
M=
\begin{pmat}[{|...|.....}]
4& 1 &1& 1& 1& 0& 0& 0 &0 &0& 0\cr\-
0& 2& 0& 0& 0& 1& 1& 0& 1& 0& 0\cr
0& 0& 2& 0& 0& 1& 0& 1& 0& 1& 0\cr
0& 0& 0& 2& 0& 0& 1& 1 &0 &0 &1\cr
0& 0& 0& 0& 2& 0& 0& 0& 1& 1& 1\cr
\end{pmat}
\end{equation*}

\begin{equation*}
E(1)=
\begin{pmat}[{|...}]
1&0& 0& 0& 0\cr\-
0&1 &1 &1 &1\cr
0&0 &1 &0 &0\cr
0&0 &0 &1 &0\cr
0&0 &0 &0 &1\cr
\end{pmat}
\end{equation*}

\begin{equation*}
E(2)=
\begin{pmat}[{|...|.....}]
1&0&0&0&0&0&0&0&0&0&0\cr\-
0&1&1&1&1&0&0&0&0&0&0\cr
0&0&1&0&0&0&0&0&0&0&0\cr
0&0&0&1&0&0&0&0&0&0&0\cr
0&0&0&0&1&0&0&0&0&0&0\cr\-
0&0&0&0&0&1&1&1&1&1&1\cr
0&0&0&0&0&1&0&1&0&1&0\cr
0&0&0&0&0&0&1&1&0&0&1\cr
0&0&0&0&0&0&0&0&1&1&1\cr
0&0&0&0&0&0&0&0&0&1&0\cr
0&0&0&0&0&0&0&0&0&0&1\cr
\end{pmat}
\end{equation*}

\begin{equation*}
E(1)\cdot M\cdot E(2)^{-1}=
\begin{pmat}[{|...|.....}]
4&1&0&0&0&0&0&0&0&0&0\cr\- 
0&2&0&0&0&2&0&0&0&0&0\cr
0&0&2&0&0&0&1&0&0&0&0\cr
0&0&0&2&0&0&0&1&0&0&0\cr
0&0&0&0&2&0&0&0&1&0&0\cr
\end{pmat}
\end{equation*}

We denote the matrix in (\ref{Jmatrix}) by $B$ and let $B'$ be the
matrix obtained by zeroing out the diagonal. Thus,

\begin{equation}\label{Jprimematrix}
B'=
\begin{pmat}[{||||||}]
 0&D_{0,1}&0&0&0&0&\hdots\cr\-
 0&0&D_{1,2}&0&0&0&\hdots\cr\-
 0&0&0&D_{2,3}&0&0&\hdots\cr\-
 0&0&0&0&D_{3,4}&0&\hdots\cr\-
 \vdots&\vdots&\vdots&\vdots&\ddots&\ddots&\ddots\cr
 0&\hdots&0&0&0&0&D_{m-1,m}\cr
\end{pmat}.
\end{equation}

We examine the matrix $D_{i-1,i}$ more closely. 
It has $\tbinom{n}{i-1}$ rows and $\tbinom{n}{i}$ columns
and has the form

\begin{equation}\label{Dij}
D_{i-1,i}=
\begin{pmat}[{||||||}]
 \mathbf{i}&\mathbf{0}&\mathbf{0}&\hdots&\mathbf{0}&\mathbf{0}&\mathbf{0}\cr\-
 \mathbf{0}&\mathbf{i-1}&\mathbf{0}&\hdots&\mathbf{0}&\mathbf{0}&\mathbf{0}\cr\-
 \mathbf{0}&\mathbf{0}&\mathbf{i-2}&\hdots&\mathbf{0}&\mathbf{0}&\mathbf{0}\cr\-
 \vdots&\vdots&\vdots&\ddots&\vdots&\vdots&\vdots\cr\-
 \mathbf{0}&\mathbf{0}&\mathbf{0}&\hdots&\mathbf{2}&\mathbf{0}&\mathbf{0}\cr\-
 \mathbf{0}&\mathbf{0}&\mathbf{0}&\hdots&\mathbf{0}&\mathbf{1}&\mathbf{0}\cr
\end{pmat},
\end{equation}
where a bold number $\mathbf{s}$, $s\neq 0$ represents a scalar matrix $sI$ 
of the appropriate size and $\mathbf{0}$ denotes a zero block of the appropriate size.
The block sizes are readily found; if $n_j=\tbinom{n}{j}-\tbinom{n}{j-1}$, then $D_{i-1,i}$ has $i+1$ block-columns and $i$ block-rows, and the $k$-th block-column of $D_{i-1,i}$ contains  $n_{k-1}$ columns, for $1\leq k\leq i+1$, while for $1\leq \ell\leq i$,
the $\ell$-th block-row contains $n_{\ell-1}$ rows. 

Also, let $\mathbf{(c)}$ denote a scalar matrix (more precisely the class of scalar matrices) whose
scalar is a multiple of $c$. (We introduce this notation because all that we shall
use about the diagonal entries is that they are even, and working with
the entire class of such matrices will facilitate the use of mathematical induction.)
Then we can write $B$ in the following form (more precisely, $B$ lies in the given class of matrices).
\begin{landscape}
\begin{equation}\label{BigJmatrix}
B=
\begin{pmat}[{|.|..|...|.|....|}]
 \mb{(2)}&\mb{1}&\mb{0}&\mb{0}&\mb{0}&\mb{0}&\mb{0}&\mb{0}&\mb{0}&\mb{0}&\hdots&\hdots&\mb{0}&\mb{0}&\hdots&\mb{0}&\mb{0}&\mb{0}&\mb{0}&\hdots&\mb{0}&\mb{0}&\mb{0}\cr\-
 \mb{0}&\mb{(2)}&\mb{0}&\mb{2}&\mb{0}&\mb{0}&\mb{0}&\mb{0}&\mb{0}&\mb{0}&\hdots&\hdots&\mb{0}&\mb{0}&\hdots&\mb{0}&\mb{0}&\mb{0}&\mb{0}&\hdots&\mb{0}&\mb{0}&\mb{0}\cr
\mb{0}&\mb{0}&\mb{(2)}&\mb{0}&\mb{1}&\mb{0}&\mb{0}&\mb{0}&\mb{0}&\mb{0}&\hdots&\hdots&\mb{0}&\mb{0}&\hdots&\mb{0}&\mb{0}&\mb{0}&\mb{0}&\hdots&\mb{0}&\mb{0}&\mb{0}\cr\-
 \mb{0}&\mb{0}&\mb{0}&\mb{(2)}&\mb{0}&\mb{0}&\mb{3}&\mb{0}&\mb{0}&\mb{0}&\hdots&\hdots&\mb{0}&\mb{0}&\hdots&\mb{0}&\mb{0}&\mb{0}&\mb{0}&\hdots&\mb{0}&\mb{0}&\mb{0}\cr
 \mb{0}&\mb{0}&\mb{0}&\mb{0}&\mb{(2)}&\mb{0}&\mb{0}&\mb{2}&\mb{0}&\mb{0}&\hdots&\hdots&\mb{0}&\mb{0}&\hdots&\mb{0}&\mb{0}&\mb{0}&\mb{0}&\hdots&\mb{0}&\mb{0}&\mb{0}\cr
 \mb{0}&\mb{0}&\mb{0}&\mb{0}&\mb{0}&\mb{(2)}&\mb{0}&\mb{0}&\mb{1}&\mb{0}&\hdots&\hdots&\mb{0}&\mb{0}&\hdots&\mb{0}&\mb{0}&\mb{0}&\mb{0}&\hdots&\mb{0}&\mb{0}&\mb{0}\cr\-
\vdots&\vdots&\vdots&\vdots&\vdots&\vdots&\ddots&\ddots&\ddots&\ddots&\ddots&\ddots&\vdots&\vdots&\vdots&\vdots&\vdots&\vdots&\vdots&\vdots&\vdots&\vdots&\vdots\cr
\vdots&\vdots&\vdots&\vdots&\vdots&\vdots&\ddots&\ddots&\ddots&\ddots&\ddots&\ddots&\vdots&\vdots&\vdots&\vdots&\vdots&\vdots&\vdots&\vdots&\vdots&\vdots&\vdots\cr
\vdots&\vdots&\vdots&\vdots&\vdots&\vdots&\ddots&\ddots&\ddots&\ddots&\ddots&\ddots&\vdots&\vdots&\vdots&\vdots&\vdots&\vdots&\vdots&\vdots&\vdots&\vdots&\vdots\cr\-
\vdots&\vdots&\vdots&\vdots&\vdots&\vdots&\vdots&\vdots&\vdots&\vdots&\ddots&\ddots&\ddots&\ddots&\ddots&\ddots&\ddots&\vdots&\vdots&\vdots&\vdots&\vdots&\vdots\cr
\vdots&\vdots&\vdots&\vdots&\vdots&\vdots&\vdots&\vdots&\vdots&\vdots&\ddots&\ddots&\ddots&\ddots&\ddots&\ddots&\ddots&\vdots&\vdots&\vdots&\vdots&\vdots&\vdots\cr\-
\mb{0}&\mb{0}&\mb{0}&\mb{0}&\mb{0}&\mb{0}&\mb{0}&\mb{0}&\mb{0}&\mb{0}&\hdots&\hdots&\mb{(2)}&\mb{0}&\hdots&\mb{0}&\mb{0}&\mb{m}&\mb{0}&\hdots&\mb{0}&\mb{0}&\mb{0}\cr
\mb{0}&\mb{0}&\mb{0}&\mb{0}&\mb{0}&\mb{0}&\mb{0}&\mb{0}&\mb{0}&\mb{0}&\hdots&\hdots&\mb{0}&\mb{(2)}&\hdots&\mb{0}&\mb{0}&\mb{0}&\mb{m-1}&\hdots&\mb{0}&\mb{0}&\mb{0}\cr
\mb{0}&\mb{0}&\mb{0}&\mb{0}&\mb{0}&\mb{0}&\mb{0}&\mb{0}&\mb{0}&\mb{0}&\hdots&\hdots&\vdots&\vdots&\ddots&\vdots&\vdots&\vdots&\vdots&\ddots&\vdots&\vdots&\vdots\cr
\mb{0}&\mb{0}&\mb{0}&\mb{0}&\mb{0}&\mb{0}&\mb{0}&\mb{0}&\mb{0}&\mb{0}&\hdots&\hdots&\mb{0}&\mb{0}&\hdots&\mb{(2)}&\mb{0}&\mb{0}&\mb{0}&\hdots&\mb{2}&\mb{0}&\mb{0}\cr
\mb{0}&\mb{0}&\mb{0}&\mb{0}&\mb{0}&\mb{0}&\mb{0}&\mb{0}&\mb{0}&\mb{0}&\hdots&\hdots&\mb{0}&\mb{0}&\hdots&\mb{0}&\mb{(2)}&\mb{0}&\mb{0}&\hdots&\mb{0}&\mb{1}&\mb{0}\cr
 \end{pmat}.
\end{equation}
\end{landscape}

\begin{lemma} $B$ and $B'$ are equivalent over the $2$-local integers $\Z_{(2)}$.
\end{lemma}
\begin{proof}
Let $\overline B(m)$ denote a matrix (actually the class of matrices) in which all the blocks (bold numbers) in $B$ are replaced by the plain numbers, and each $\mb{(c)}$ is replaced by $(c)$, representing 
a multiple of the $2$-local integer $c$. It suffices to show that any matrix of the form $B(m)$ is
$2$-locally equivalent to the matrix obtained by zeroing out its diagonal.
Indeed, if there is a sequence of $\Z_{(2)}$-unimodular row and column operations
which zero out the diagonal of $\overline B(m)$, then the same sequence
of the blockwise versions of these operations will kill the diagonal of $B$.
So we may discard $B$ and  work only  with $\overline B(m)$ from now on.

\begin{equation}\label{Jbar}
\overline B(m)=
\begin{pmat}[{|||||||}]
 (2)I&\overline D_1&0&0&0&\hdots&0&0\cr\-
 0&(2)I&\overline D_2&0&0&\hdots&0&0\cr\-
 0&0&(2)I&\overline D_3&0&\hdots&0&0\cr\-
 0&0&0&(2)I&\overline D_4&\hdots&0&0\cr\-
 \vdots&\vdots&\vdots&\ddots&\ddots&\ddots&\vdots&\vdots\cr\-
 0&\hdots&\hdots&\hdots&\ddots&(2)I&\overline D_{m-1}&0\cr\-
 0&\hdots&\hdots&\hdots&\hdots&0&(2)I&\overline D_{m}\cr
 \end{pmat},
\end{equation}
where
\begin{equation}
\overline D_{i}=
\begin{pmat}[{.....|}]
 {i}&{0}&{0}&\hdots&{0}&{0}&{0}\cr
 {0}&{i-1}&{0}&\hdots&{0}&{0}&{0}\cr
 {0}&{0}&{i-2}&\hdots&{0}&{0}&{0}\cr
 \vdots&\vdots&\vdots&\ddots&\vdots&\vdots&\vdots\cr
 {0}&{0}&{0}&\hdots&{2}&{0}&{0}\cr
 {0}&{0}&{0}&\hdots&{0}&{1}&{0}\cr
\end{pmat}
\end{equation}
is the ``condensed'' version of $D_{i-1,i}$,  consisting of a diagonal $i\times i$ matrix augmented by a column of zeros.

From now on it will be convenient to refer to the entries of $\overline B(m)$ (and
matrices derived from it) by their positions relative to the submatrices
$\overline D_i$.  Note this is a change in the way we are indexing blocks, compared
to how we did it in the original matrix $M$.  
 The row of the main matrix containing
the $k$-th  row of $\overline D_i$ is assigned the label $[i,k]$
and the column containing the $\ell$-th column of $\overline D_{j}$
is assigned the label $[j,\ell]$, while the first column is
labeled $[0,1]$. Thus, in a row index $[i,k]$ we have $1\leq i\leq m$
and $1\leq k\leq i$, while a column index $[j,\ell]$ has
$0\leq j\leq m$ and $1\leq \ell\leq j+1$.

We shall perform some row and column operations
using the odd entries on the main diagonals of the $\overline D_i$s to kill
diagonal entries of the main matrix. Each odd entry of  $\overline D_i$
for $1\leq i\leq m-1$ will be used to kill the two diagonal entries of the main matrix, one in the same row as the odd entry and one  in the same column. The  odd entries of $\overline D_m$ will be used to kill the diagonal entry in the same row.  This procedure will create  new entries $(4)$ at locations where  there previously were zeroes. 
To be precise, if the odd entry is at  $([i,k],[i,k])$  then the diagonal entry of the main matrix in  the same row is at $([i,k],[i-1,k])$. The diagonal entry of the main matrix in  the same column as $([i,k],[i,k])$ is at $([i+1,k],[i,k])$ (with no such
entry if $i=m$). By multiplying column $[i,k]$
by a suitable element of $(2)$  and subtracting it from column $[i-1,k]$, we
kill off the  entry at $([i,k],[i-1,k])$  and create a new entry $(4)$ at $([i+1,k],[i-1,k)$,
if $i\leq m-1$. The new entry is $(4)$ because the 
entry being subtracted from zero is a $(2)$-multiple of the $(2)$ at $([i+1,k],[i,k])$
on the main diagonal. No new entry is created if $i=m$.
Then, we can subtract a $(2)$-multiple of  row $[i,k]$ from row $[i+1,k]$ to kill off the
diagonal entry at $([i+1,k],[i,k])$, without creating any new nonzero entries, if
$i\leq m-1$, while there is nothing to be done when $i=m$.
The following figure (for $m=5$) shows the resulting matrix after performing these operations with the entry 1 at  $([3,3],[3,3])$. The block indices are specified by numbered braces.

\begin{equation*}
\begin{array}{rl}
&\begin{matrix}
\overbrace{\phantom{\begin{pmat}.{}.(0)\cr \end{pmat}}}^{\mbox{0}}&
\overbrace{\phantom{\begin{pmat}.{}.(0)&0\cr\end{pmat}}}^{\mbox{1}}&\overbrace{\phantom{\begin{pmat}.{}.(0)&(0&0\cr\end{pmat}}}^{\mbox{2}}&\overbrace{\phantom{\begin{pmat}.{}.(0)&(0)&0&0\cr\end{pmat}}}^{\mbox{3}}&\overbrace{\phantom{\begin{pmat}.{}.(0)&(0)&(0)&(0&0\cr\end{pmat}}}^{\mbox{4}}&
\overbrace{\phantom{\begin{pmat}.{}.0&0&0&0&(0)\cr\end{pmat}}}^{\mbox{5}}
\end{matrix}\\
\begin{pmat}.{}.
\coolleftbrace{1}{Y}\cr
\coolleftbrace{2}{Y\\Y}\cr
\coolleftbrace{3}{Y\\ Y\\Y}\cr
\coolleftbrace{4}{Y\\Y\\Y\\Y}\cr
\coolleftbrace{5}{Y\\ Y\\ Y\\ Y\\ Y\\Y}\cr
\end{pmat}%
&
\begin{pmat}[{|.|..|...|....|.....}]
(2)&1&0&0&0&0&0&0&0&0&0&0&0&0&0&0&0&0&0&0&0\cr\-
0&(2)&0&2&0&0&0&0&0&0&0&0&0&0&0&0&0&0&0&0&0\cr
0&0&(2)&0&1&0&0&0&0&0&0&0&0&0&0&0&0&0&0&0&0\cr\-
0&0&0&(2)&0&0&3&0&0&0&0&0&0&0&0&0&0&0&0&0&0\cr
0&0&0&0&(2)&0&0&2&0&0&0&0&0&0&0&0&0&0&0&0&0\cr
0&0&0&0&0&0&0&0&1&0&0&0&0&0&0&0&0&0&0&0&0\cr\-
0&0&0&0&0&0&(2)&0&0&0&4&0&0&0&0&0&0&0&0&0&0\cr
0&0&0&0&0&0&0&(2)&0&0&0&3&0&0&0&0&0&0&0&0&0\cr
0&0&0&0&0&(4)&0&0&0&0&0&0&2&0&0&0&0&0&0&0&0\cr
0&0&0&0&0&0&0&0&0&(2)&0&0&0&1&0&0&0&0&0&0&0\cr\-
0&0&0&0&0&0&0&0&0&0&(2)&0&0&0&0&5&0&0&0&0&0\cr
0&0&0&0&0&0&0&0&0&0&0&(2)&0&0&0&0&4&0&0&0&0\cr
0&0&0&0&0&0&0&0&0&0&0&0&(2)&0&0&0&0&3&0&0&0\cr
0&0&0&0&0&0&0&0&0&0&0&0&0&(2)&0&0&0&0&2&0&0\cr
0&0&0&0&0&0&0&0&0&0&0&0&0&0&(2)&0&0&0&0&1&0\cr
\end{pmat}
\end{array}
\end{equation*}

We do this for all odd entries of all the $\overline D_i$. In our $m=5$ example,
after the operations the matrix looks as follows.
\begin{equation*}
\begin{pmat}[{|.|..|...|....|.....}]
0&1&0&0&0&0&0&0&0&0&0&0&0&0&0&0&0&0&0&0&0\cr\-
(4)&0&0&2&0&0&0&0&0&0&0&0&0&0&0&0&0&0&0&0&0\cr
0&0&0&0&1&0&0&0&0&0&0&0&0&0&0&0&0&0&0&0&0\cr\-
0&0&0&0&0&0&3&0&0&0&0&0&0&0&0&0&0&0&0&0&0\cr
0&0&(4)&0&0&0&0&2&0&0&0&0&0&0&0&0&0&0&0&0&0\cr
0&0&0&0&0&0&0&0&1&0&0&0&0&0&0&0&0&0&0&0&0\cr\-
0&0&0&(4)&0&0&0&0&0&0&4&0&0&0&0&0&0&0&0&0&0\cr
0&0&0&0&0&0&0&0&0&0&0&3&0&0&0&0&0&0&0&0&0\cr
0&0&0&0&0&(4)&0&0&0&0&0&0&2&0&0&0&0&0&0&0&0\cr
0&0&0&0&0&0&0&0&0&0&0&0&0&1&0&0&0&0&0&0&0\cr\-
0&0&0&0&0&0&0&0&0&0&0&0&0&0&0&5&0&0&0&0&0\cr
0&0&0&0&0&0&0&(4)&0&0&0&0&0&0&0&0&4&0&0&0&0\cr
0&0&0&0&0&0&0&0&0&0&0&0&0&0&0&0&0&3&0&0&0\cr
0&0&0&0&0&0&0&0&0&(4)&0&0&0&0&0&0&0&0&2&0&0\cr
0&0&0&0&0&0&0&0&0&0&0&0&0&0&0&0&0&0&0&1&0\cr
\end{pmat}
\end{equation*}

At this point, the rows and columns of the main matrix that correspond to an odd
entry of some $\overline D_i$ have no other nonzero entries. The submatrix $D(m)$ 
formed from these rows and columns is a square diagonal matrix with odd entries,
and the main matrix is the block sum of this diagonal matrix with the submatrix $A(m)$
formed from the remaining rows and columns. (Since permuting rows and columns results
in an integrally equivalent matrix, the reader may find it helpful, for easier 
visualization of this block sum decomposition, to renumber the rows and columns in order to put $D(m)$ at the bottom right of the main matrix.)
We observe that $A(m)$ will have entries $(4)$ all along its main diagonal,
and each row of $A(m)$ has one nonzero entry off the main diagonal (coming from the
even  entries on the main diagonals of the $\overline D_i$.)
In the picture below, we show $A(5)$.

\begin{equation*}
\begin{pmat}[{||.|.|..|..}]
(4)&0&2&0&0&0&0&0&0&0&0&0\cr\-
0&(4)&0&0&2&0&0&0&0&0&0&0\cr\-
0&0&(4)&0&0&0&4&0&0&0&0&0\cr
0&0&0&(4)&0&0&0&2&0&0&0&0\cr\-
0&0&0&0&(4)&0&0&0&0&4&0&0\cr
0&0&0&0&0&(4)&0&0&0&0&2&0\cr
\end{pmat}
\end{equation*}
 
We next note that if $i$ and $j$ have different parity then
all entries at locations $([i,k],[j,\ell])$ are zero, 
so $A(m)$ is the block sum of the submatrix $A''(m)$ corresponding to
rows $[i,k]$ and columns $[j,\ell]$ for odd indices $i$ and $j$,
and the submatrix $A'(m)$ formed from the rows and columns corresponding to the even
indices. For better visualization, we can reorder the rows and columns.
The block decomposition of $A(5)$ after reording rows and columns is
shown below.

\begin{equation*}
\begin{pmat}[{|.|..||.|..}]
(4)&2&0&0&0&0&0&0&0&0&0&0\cr\-
0&(4)&0&4&0&0&0&0&0&0&0&0\cr
0&0&(4)&0&2&0&0&0&0&0&0&0\cr\-
0&0&0&0&0&0&(4)&2&0&0&0&0\cr\-
0&0&0&0&0&0&0&(4)&0&4&0&0\cr
0&0&0&0&0&0&0&0&(4)&0&2&0\cr
\end{pmat}
\end{equation*}
 
We can see from this picture that $A'(5)=2\overline B(2)=A''(5)$. 
In general, it is easy to see  that  if we delete the rows and columns
of $\overline D_i$ which contain odd entries, the resulting matrix is
$2\overline D_{i'}$ where $i'=\lfloor\frac{i}{2}\rfloor$ (with
the understanding that $\overline D_{0}$ is the empty matrix).
It follows that $A'(m)=2\overline B(m')$, where $m'=\lfloor\frac{m}{2}\rfloor$,
and $A''(m)=2\overline B(m'')$, and  $m''=\lfloor\frac{m-1}{2}\rfloor$.  (More formally, $A'(m)$ belongs to the class $2B(m')$ and $A''(m)$ belongs to the class $2B(m'')$.)

Thus, we have shown that 
$\overline B(m)$ is $2$-locally equivalent to a matrix $C$ which is the diagonal block sum
of $2\overline B(m')$ and $2\overline B(m'')$ and the diagonal matrix $D(m)$. 
Arguing by induction on $m$, there are $\Z_{(2)}$-unimodular row and column operations 
on $C$  that kill the diagonals of $A'(m)$ and of $A''(m)$ while leaving all other entries
of $C$  unchanged. The resulting matrix is equal, up to reordering  rows and columns, to the matrix obtained from $\overline B(m)$ by zeroing out its diagonal.
\end{proof}

\begin{theorem}\label{Adiag} Suppose that $n=2m$ is even. Then the adjacency matrix
 $A$ of the $n$-cube $Q_n$ is $\Z_{(2)}$-equivalent to a diagonal form 
 with $\tbinom{n}{m}$ diagonal entries equal to zero and whose nonzero diagonal entries are the integers $k=1$,$2$,\dots $m$, in which the multiplicity of $k$ is $2\tbinom{n}{m-k}$. 
\end{theorem}

\begin{proof} By Lemma~\ref{half}, it suffices to find a diagonal form for $M$,
hence for $B'$, with which we have shown $M$ to be integrally equivalent. 
The nonzero entries of $B'$ are the integers $1$ to $m$. The integer $k$
occurs in $D_{i-1,i}$ only when $k\leq i$ and then its
multiplicity in  $D_{i-1,i}$ is $\tbinom{n}{i-k}-\tbinom{n}{i-1-k}$.
Therefore, the multiplicity of $k$ in $B'$ is 
\begin{equation}
\sum_{i=k}^{m}\tbinom{n}{i-k}-\tbinom{n}{i-1-k} = 
\sum_{s=0}^{m-k}\tbinom{n}{s}-\tbinom{n}{s-1}=\tbinom{n}{m-k}.
\end{equation}
\end{proof}

\begin{corollary} Conjecture~\ref{conjDJ} is true.
\end{corollary}
By Lemma~\ref{half}, the multiplicity of the prime power 
$p^e$ as an $p$-elementary divisor of $\tilde A$, and hence also of  
$A$, is twice the sum of the binomial coefficients  $\tbinom{n}{m-k}$, taken over
those $k$ with $1\leq k\leq m$ that are exactly divisible by $p^e$.
On the other hand, we know (\S~\ref{ncube}) 
that the nonzero eigenvalues of $A$ are the integers $\pm 2k$ for $1\leq k\leq m$,
and the multiplicity of $2k$ is $\tbinom{n}{m-k}$. 
Comparing these numbers for $p=2$, we see that Conjecture~\ref{conjDJ} is true.

\subsection{Proof of Theorem~\ref{Qdiag}}
By Theorem~\ref{Adiag}, it suffices to show, for every odd prime $p$, that $A$ is $\Z_{(p)}$-equivalent to a diagonal matrix whose nonzero entries are $k=1$,\dots,$m$, where
$k$ has multiplicity $2\tbinom{n}{m-k}$.  Let $p$ be given. We know from \cite{DJ} that $A$ is
$\Z_{(p)}$-equivalent to a diagonal matrix whose nonzero are $n-2\ell$
with multiplicity $\tbinom{n}{\ell}$, for $0\leq \ell\leq n$.  The latter is easily seen to be
integrally equivalent to a diagonal matrix whose nonzero entries are $2k=1$,\dots,$m$, where
$k$ has multiplicity $2\tbinom{n}{m-k}$, and  hence $\Z_{(p)}$-equivalent to
the diagonal form given in Theorem~\ref{Qdiag}.

\subsection{Final remarks}
It would be of interest to find a diagonal form for the Laplacian matrix $nI-A$ of $Q_n$.
The Smith group of this matrix is called the {\it critical group} of $Q_n$.
By the results of \cite{Bai}, only the $2$-Sylow subgroup of the critical group
remains to be determined, for both odd and even $n$. 
We do not have any conjecture about its exact structure.  
However, we note that if two integral matrices are equal modulo $p^s$ then for  $i<s$
the  multiplicity of $p^i$ as a $p$-elementary divisor is the same for both matrices.
Thus, for example, when $n=2^s$, Theorem~\ref{Adiag} gives part of the cyclic
decomposition of the $2$-Sylow subgroup of the critical group.


\begin{thebibliography}{99}
\bibitem{Bai}  Hua Bai, On the critical group of the n-cube. Linear Algebra Appl. \textbf{ 369} (2003)  251-–261.

\bibitem{DJ} J. Ducey, D. Jalil, Integer invariants of abelian Cayley graphs, Linear Algebra Appl. \textbf{445} (2014) 316--325

\bibitem{Bier}
Thomas Bier, \emph{Remarks on recent formulas of {W}ilson and {F}rankl},
European J. Combin. \textbf{14} (1993), no.~1, 1--8. 

\bibitem{Frankl}
P. Frankl, \emph{Intersection theorems and mod p rank of inclusion matrices},
Journal of Combinatorial Theory, Series A, \textbf{54} (1990), 85–-94

\bibitem{Wilson}
Richard~M. Wilson, \emph{A diagonal form for the incidence matrices of
  {$t$}-subsets vs.\ {$k$}-subsets}, European J. Combin. \textbf{11} (1990),
  no.~6, 609--615. 


\end{thebibliography}
\end{document}